\documentclass[12pt,twoside]{amsart}
\usepackage{amssymb,amsmath,amsthm, amscd, enumerate, mathrsfs, upgreek}
\usepackage{graphicx, hhline, tikz}
\usepackage[colorlinks=true,pagebackref,hyperindex]{hyperref}
\usepackage[all]{xy}

\usepackage{url}

\usepackage{color}
\usepackage[backrefs, alphabetic, initials]{amsrefs}
\usepackage{color}  
\usepackage{latexsym, bm}
\usepackage[T1]{fontenc} 
\usepackage{fancyhdr}
\usepackage{enumerate}
\usepackage{enumitem}
\title[Strictly nef divisors]
{On the log version of Serrano's conjecture}

\author{Haidong Liu}

\date{\today, version 0.02}
\address{Sun Yat-sen University, Department of mathematics, Guangzhou, 510275, China}
\email{liuhd35@mail.sysu.edu.cn, jiuguiaqi@gmail.com}
\urladdr{\url{https://sites.google.com/view/liuhaidong}}

\DeclareMathOperator{\nklt}{Nklt}
\DeclareMathOperator{\ch}{ch}

\DeclareMathOperator{\alb}{Alb}

\newtheorem{thm}{Theorem}[section]
\newtheorem{lem}[thm]{Lemma}
\newtheorem{prop}[thm]{Proposition}
\newtheorem{conj}[thm]{Conjecture}
\newtheorem{cor}[thm]{Corollary}

\theoremstyle{definition}

\newtheorem{defn}[thm]{Definition}
\newtheorem{rem}[thm]{Remark}
\newtheorem*{ack}{Acknowledgments}

\makeatletter
    
    \@addtoreset{equation}{section}
\makeatother
\begin{document}

\begin{abstract}
In this paper, we continue the study of Serrano's conjecture in low dimensions.
We focus on two special cases of the log version of Serrano's conjecture: the ampleness conjecture and the log version of Campana--Peternell's conjecture.
In dimension 3, we prove that the ampleness conjecture holds for
non-canonical singularities;
by the same method, we also prove that the log canonical version of 
Campana--Peternell's conjecture holds in dimension 3.
In dimension 4, we improve the results on Campana--Peternell's conjecture by excluding the case that the numerical dimension of the anti-canonical divisor is 3. 
Specifically, we show that for a projective
smooth fourfold $X$, if $-K_X$ is strictly nef but not ample, 
then $\kappa(X, -K_X)=0$ and 
$\nu(X, -K_X)=2$; in this case, if we further assume that $X$ admits a Fano contraction $X\to Y$ onto a surface $Y$ induced by some extremal ray, then $\rho(X)=2$.
\end{abstract}

\maketitle 

\tableofcontents

\section{Introduction}\label{sec1}

In this paper, we continue the studies \cites{liu, liu-matsumura, liu-svaldi} of Serrano's conjecture in dimensions 3 and 4. 
Serrano’s original conjecture \cite{serrano}*{Question 0.1} predicts that on a projective manifold, 
a small deformation
of a strictly nef divisor in the direction of the canonical divisor is ample. 
This is a weak analogue of Fujita’s conjecture \cite{laz}*{\S 10.4.A} for strictly nef divisors.
From the viewpoint of the minimal model program, it is very natural to generalize Serrano's conjecture
to mild singularities, parallel to the log version of Fujita's conjecture.
The most general statement of the log version of Serrano's conjecture, 
as far as the author knows, was formulated in \cite{han-liu}*{Question 3.5} from the aspect of generalized log canonical pairs.
However, the effective bound predicted there is related to another conjecture on
the length of extremal rays, which is far from reaching. Therefore, we present here a more workable version taken
from \cite{liu-matsumura}*{Conjecture 1.5}:

\begin{conj}[Log version of Serrano's conjecture]\label{conj.serrano.lv}
Let $(X,\Delta)$ be a projective log canonical pair of dimension $n$ and $L$ be a strictly nef Cartier divisor on $X$.
Then, the $\mathbb Q$-Cartier divisor $K_X+\Delta+tL$ is ample for $t>2n$.
\end{conj}

There are three important special cases of the log version of Serrano's conjecture:  $(i)$ $K_X$ is numerically trivial; $(ii)$ $L= -K_X$ is strictly nef; $(iii)$ $L= K_X$ is strictly nef. We pay attention to the first two cases in this paper.

\subsection{The case $K_X\equiv 0$}
In this case, the log version of Serrano's conjecture is also called \emph{ampleness conjecture} when $\Delta=0$:

\begin{conj}[Ampleness conjecture]\label{conj.amp}
Let $X$ be a projective log canonical variety with numerically trivial canonical divisor. 
Then, any strictly nef $\mathbb Q$-Cartier divisors on $X$ are ample.
\end{conj}

As the smooth case \cite{liu-matsumura}*{Conjecture 1.1}, Conjecture \ref{conj.amp} is also a special case of the semiampleness conjecture,
which is well-known to hold true for surfaces in characteristic zero, and is proved recently to hold true for klt surfaces in characteristic $p>0$  by \cite{bern-stig}. 
In 3-dimensional smooth case, \cites{ccp,serrano} proved that Conjecture \ref{conj.amp} holds unless possibly $H^1(X,\mathcal O_X)=0$ and 
$L^3=c_2(X)\cdot L=0$ for the strictly nef divisor $L$; later, the author and Svaldi \cite{liu-svaldi} improved their results a little bit from the study of the semiampleness conjecture in dimension 3. 
In 4-dimensional smooth case,
 the author and Matsumura \cite{liu-matsumura} proved the ampleness conjecture 
for fourfolds without irregularity.
In Section \ref{sec3}, we investigate 3-dimensional singular case and obtain the following result:

\begin{thm}[Corollary \ref{cor.nc}]\label{thm.main.dim3}
    Let $X$ be a projective log canonical variety of dimension 3 with numerically trivial canonical divisor. Assume that $X$ has non-canonical singularities. Then, any strictly nef $\mathbb Q$-Cartier divisors on $X$ are ample.
\end{thm}

By Theorem \ref{thm.main.dim3}, Theorem \ref{thm.irregular} and Lemma \ref{lem.Q}, we see that the remaining case of the ampleness conjecture in dimension 3 is that $X$ is a canonical Calabi--Yau threefold, 
that is, a projective canonical variety $X$ of dimension 3 
such that $K_X\sim 0$ and $H^1(X,\mathcal O_X)=0$ (see Definition \ref{defn.ccy}). In this case, we obtain some partial results:


\begin{thm}[Theorems \ref{thm.nu2.dim3} and \ref{thm.nu1.dim3}]
Let $X$ be a $\mathbb Q$-factorial canonical Calabi--Yau threefold
and $L$ be a strictly nef $\mathbb Q$-Cartier divisor on $X$.
\begin{enumerate}
    \item if $\nu(X, L)=2$, then $H^p(X, \Omega^{[q]}_X(mL))=0$ for all $p, q\geq 0$ and $m\gg 1$;
    \item if $\nu(X, L)=1$ and $X$ admits a fibration $f\colon X\to S$, then $f$ is an equiv-dimensional elliptic fibration and $S$ is a log terminal Fano surface with $\rho(S)=1$.
\end{enumerate} 
\end{thm}

In the proofs of above theorems, we use the inductive strategy as in \cites{ccp, liu, liu-matsumura, serrano} and the references therein.
 Note that in this inductive strategy, the log version of the ampleness conjecture with a nonzero boundary $\Delta$ plays an important role, 
 which is established in dimension 3 (see Subsection \ref{subsec.lv} and Proposition \ref{prop.sn}).

\subsection{The case $L=-K_X$ is strictly nef}
In this case, the log version of Serrano's conjecture is also a singular version of \emph{Campana--Peternell's conjecture} when $\Delta=0$:

\begin{conj}[Log version of Campana--Peternell's conjecture]\label{conj.cp}
Let $X$ be a projective log canonical variety
such that $-K_X$ is strictly nef.
Then, $-K_X$ is ample.
\end{conj}

The smooth version of Campana--Peternell's conjecture
has been confirmed in dimension 2 by Maeda \cite{maeda} and in dimension 3
by Serrano \cite{serrano}. For the singular version in dimension 3, Uehara \cite{uehara} proved it for canonical singularities and Liu--Ou--Yang--Wang--Zhong
\cite{loywz} proved it for klt singularities. When $\Delta=0$, we sightly generalize 
\cite{loywz}*{Theorem D} to log canonical singularities by a different approach with \cite{loywz}.

\begin{thm}[Corollary \ref{cor.cp3}]
Let $X$ be a projective log canonical variety of dimension 3
such that $-K_X$ is strictly nef.
Then,  $-K_X$ is ample. 
\end{thm}

In dimension 4, the author proved the smooth version of Campana--Peternell's conjecture in \cite{liu}, except 
the case that $c^2_1(X)\cdot c_2(X)=0$ and 
$-K_X$ is linearly equivalent to some prime Calabi--Yau divisor $V$ (see Definition \ref{defn.pcy}). 
In Section \ref{sec4}, we further rule out the case $\nu(X, -K_X)=3$.
More precisely, 
combining with Corollary \ref{cor.kd0}, Theorems \ref{thm.nm3} and \ref{thm.fc}, we obtain the following result:

\begin{thm}\label{thm.main.dim4}
Let $X$ be a projective smooth fourfold such that $-K_X$ is strictly nef. If $-K_X$ is not ample, then 
\begin{enumerate}
    \item $c^2_1(X)\cdot c_2(X)=0$;
    \item $\nu(X, -K_X)=2$;
    \item $\kappa(X, -K_X)=0$ and $-K_X\sim V$, where $V$ is a prime Calabi--Yau divisor.
\end{enumerate}
Moreover, if $X$ admits a Fano contraction $f\colon X\to Y$ onto a surface $Y$ induced by some extremal ray, then $\rho(X)=2$
and $Y$ is a log terminal Fano surface with $\rho(Y)=1$.
\end{thm}

\begin{rem}
In this paper, we drop the boundary $\Delta$ just for simplicity. 
There is no difficulty to generalize our results to the case $\Delta\neq 0$. Actually the case $\Delta\neq 0$ should be easier, as showed
in Subsection \ref{subsec.lv} and Section \ref{sec3}.
\end{rem}

\begin{ack}
The author would like to thank Chen Jiang and Guolei Zhong for useful discussions and suggestions.
\end{ack}

Throughout this paper,  we work over the complex number field $\mathbb C$.
A {\em scheme} is always assumed to be separated and of finite type over $\mathbb{C}$, 
and a {\em variety} is a reduced and irreducible algebraic scheme. 
We will freely use the basic notation 
in \cites{fujino-foundations, kmm, kollar-mori, laz}.

\section{Preliminaries}\label{sec2}

In this section, we present preliminary results.

\subsection{Iitaka dimension and Numerical dimension}
Let $D$ be a nef $\mathbb Q$-Cartier $\mathbb Q$-divisor 
on a projective normal variety $X$. 
Let $m_0$ be a positive integer 
such that $m_0D$ is a Cartier divisor. 
Let $$\Phi_{|mm_0D|}\colon X\dashrightarrow \mathbb P^{\dim\!|mm_0D|}$$ be 
the rational map given by the complete linear system $|mm_0D|$ 
for a positive integer $m$. 
Note that $\Phi_{|mm_0D|}(X)$ denotes the 
closure of the image of the rational map $\Phi_{|mm_0D|}$. 
We put 
$$\kappa (X, D):=\max_m \dim \Phi_{|mm_0D|}(X)$$ if 
$|mm_0D|\ne \emptyset$ for some $m$ and 
$\kappa (X, D):=-\infty$ otherwise. 
We call $\kappa(X, D)$ the {\em{Iitaka dimension}} of $D$.

The {\em{numerical dimension}} 
of $D$ on $X$ is defined by 
\[
\nu(X, D):=\max\{h\in \mathbb N \; |  \;D^h\not\equiv 0\}.
\]

It is well-known that 
$\max\{\kappa (X, D), 0\}\leq \nu(X, D) \leq \dim X$,
that $\nu(X, D)=\dim X$ if and only if $D$ is big, 
and that the numerical dimension is invariant under 
pulling back by proper surjective morphisms.

\subsection{Singularities of pairs}
Let $X$ be a normal variety such that $K_X$ is $\mathbb Q$-Cartier.
Let $f\colon Y\to X$ be a resolution. Then, we can write
\[
K_Y=f^*K_X+\sum _E a(E,X) E,
\]
where $E$ runs over all prime divisors on $Y$.
If $a(E, X)> 0$ (resp. $a(E, X)\geq 0$ , $a(E, X)> -1$  and $ a(E, X)\geq -1$) for every prime divisor $E$ over $X$,
then $X$ is called \emph{terminal} (resp. \emph{canonical}, \emph{log terminal} and \emph{log canonical}).
Without the boundary, log terminal is equivalent to Kawamata log terminal, 
which is usually denoted by \emph{klt} for short.

If there exists a resolution
$f:Y\to X$ and a prime divisor $E$ on $Y$
with $a(E, X)\leq -1$, then $f(E)$ is called a \emph{non-klt center} of $X$. The closed subscheme defined by the ideal sheaf 
$f_*\mathcal O_Y(-\lfloor -\sum a(E,X) E \rfloor )\subseteq \mathcal O_X$ 
is called the \emph{non-klt locus} of $X$, denoted by $\nklt(X)$.

\subsection{Canonical Calabi--Yau threefolds and Prime Calabi--Yau divisors}\label{subsec.cy}
The following fact should be well-known to experts. Its proof is sketched in the third paragraph of the proof of \cite{liu-matsumura}*{Theorem 3.5}, 
where \cite{oguiso}*{Proposition 2.7} is used; 
however, the simple connectedness is needed in \cite{oguiso}*{Proposition 2.7}. 
So we follow \cites{lop, lp} instead and give a detailed proof for the reader's convenience. 

\begin{lem}\label{lem.Q}
Let $X$ be a projective canonical variety of dimension 3 such that $K_X\equiv 0$ and $H^1(X,\mathcal O_X)=0$. 
If $K_X\not\sim 0$, then any nef $\mathbb Q$-Cartier divisors on $X$ are semiample.
In particular, any strictly nef $\mathbb Q$-Cartier divisors on $X$ are ample.
\end{lem}

\begin{proof}
By the abundance theorem for canonical varieties, $K_X\equiv 0$ if and only if $K_X\sim_{\mathbb Q} 0$ (we can also get this equivalence by $H^1(X,\mathcal O_X)=0$ and the exponential sheaf sequence directly).
Replacing $X$ by its terminalization, we can assume that $X$ is $\mathbb Q$-factorial and terminal, where $H^1(X,\mathcal O_X)=0$ and $K_X\not\sim 0$ by assumptions.
Therefore, $H^3(X, \mathcal O_X)=H^0(X, K_X)=0$ by Serre's duality. It follows that
\[
\chi(\mathcal O_{X})=\sum^3_{i=0}(-1)^ih^i(X, \mathcal O_X)=h^0(X, \mathcal O_X)+h^2(X, \mathcal O_X)\geq 1.
\]
Let $L$ be a nef Cartier divisor on $X$.
Then by \cite{oguiso}*{Lemma 1.4} (or Reid's formula \cite{reid}), we obtain
\begin{equation}\label{eq.chi}
\chi(X, mL)=\frac{1}{6}(mL)^3+\frac{1}{12}c_2(X)\cdot mL+
\chi(\mathcal O_{X})\geq 1
\end{equation}
for any $m\in \mathbb N$. 
If $\nu(X, L)=2$, then $H^2(X, mL)=0$ for $m \geq 1$ 
by the Kawamata--Viehweg vanishing. It follows
that $h^0(X, mL)\geq \chi(X, mL)\geq 1$ for $m\geq 1$, that is,
$\kappa(X, L)\geq 0$.

Assume that $\nu(X, L)=1$. 
Let $f\colon Y\to X$ be a log resolution. 
By \cite{lp}*{Theorem 6.5}, we have 
either $\kappa(X, L)\geq 0$, or the multiplier ideal sheaf 
$\mathcal I(h^{\otimes m})$ for any $m\geq 1$ and
any singular metric $h$ on $f^*L$
with semipositive curvature current, is linearly equivalent to $\mathcal O_Y$.
In particular, 
$L$ has \emph{algebraic singularities} in the sense of \cite{lp}*{Subsection 2.D}.
Therefore, \cite{lp}*{Theorem F(i)} implies that $\kappa(X, L)\geq 0$.

In any cases, we obtain that $\kappa(X, L)\geq 0$,
that is, $L$ is a nef and effective $\mathbb Q$-divisor on $X$. By the abundance theorem
for canonical threefolds, we have that
$L$ is semiample. In particular, there exists an induced morphism $g\colon X\to Z$ and an ample $\mathbb Q$-divisor $H$ on $Z$ such that 
$L=g^*H$. If $L$ is strictly nef, then $g$ has to be finite, and hence $L$ is ample.
\end{proof}

Thanks to Lemma \ref{lem.Q}, we can simplify the notation for the purpose of this paper:

\begin{defn}\label{defn.ccy}
A \emph{Calabi--Yau variety} is a projective normal variety $X$
such that $K_X\sim 0$ and $H^1(X,\mathcal O_X)=0$. 
A \emph{canonical Calabi--Yau threefold} is a Calabi--Yau variety of dimension 3 with at worst canonical singularities. 
\end{defn}

\begin{rem}
There are lots of definitions for ``Calabi--Yau'' varieties. Some assumed that $K_X\equiv 0$; some are allowed to be irregular, that is, $H^1(X,\mathcal O_X)\neq 0$;
some are restricted to that $H^i(X,\mathcal O_X)= 0$ for $0<i<\dim X$;
some are restricted to have mild (e.g., at worst log canonical) singularities; 
some are restricted to be simply connected. 
We will see later (after Theorems \ref{thm.nc} and \ref{thm.irregular}) that Definition \ref{defn.ccy} is sufficient to study the ampleness conjecture in dimension 3. In dimension 4,
it is also called \emph{K-trivial} sometimes (see \cite{liu-matsumura} for example).  
\end{rem}

When $X$ is a canonical Calabi--Yau threefold, we see that 
\begin{equation*}
    \begin{aligned}
        H^2(X, \mathcal{O}_X)&=H^1(X, \mathcal{O}_X(K_X))=H^1(X, \mathcal{O}_X)=0\\
        H^3(X, \mathcal{O}_X)^*&=H^0(X, \mathcal{O}_X(K_X))=H^0(X, \mathcal{O}_X)\cong \mathbb C
    \end{aligned}
\end{equation*}
by Serre's duality. In particular, $\chi(\mathcal O_X)=0$. This is the biggest trouble to use Hirzebruch--Riemann--Roch formula to create a global section for a nef or strictly nef divisor on a canonical Calabi--Yau threefold, hence it is extremely difficult to prove the generalized abundance conjecture or the ampleness conjecture in this case.

Using Definition \ref{defn.ccy}, the \emph{prime Calabi--Yau divisor} defined in  
\cite{liu-matsumura}*{Subsection 2.4} can be rephrased as follows:

\begin{defn}\label{defn.pcy}
a prime divisor $D$ on a normal variety $X$ is said to be a \emph{prime Calabi--Yau divisor} if $D$ is a canonical Calabi--Yau variety.
\end{defn}

We refer to \cites{liu, liu-matsumura} for more details about prime Calabi--Yau divisors.

\subsection{Results on the log version of Serrano's conjecture}\label{subsec.lv}
As stated at the end of \cite{liu-matsumura}*{Section 3}, if the boundary $\Delta$ in Conjecture \ref{conj.serrano.lv} is not zero, then we can 
perform induction of the dimension onto $\Delta$. By this method, we proved a log version of the ampleness conjecture with a boundary in dimension $\leq 4$, and proved a 
partial result on \cite{liu}*{Conjecture 2.3} in dimension 4:

\begin{thm}[\cite{liu-matsumura}*{Theorem 1.6}]\label{thm.delta.3}
Let $X$ be a Calabi--Yau manifold of dimension $n\leq 4$, 
$\Delta$ be a nonzero effective $\mathbb Q$-divisor, and $L$ be a strictly nef $\mathbb Q$-Cartier divisor on $X$. 
Then, the $\mathbb Q$-Cartier divisor $\Delta+tL$ is ample for $t\gg 1$.
\end{thm}

\begin{thm}[\cite{liu}*{Theorem 3.1}]\label{thm.delta.4}
Let $X$ be a projective smooth fourfold such that $-K_X$ is strictly nef.
Let $V$ be a nonzero prime divisor on $X$. 
If one of the following 
\begin{enumerate}
 \item $V$ is not a prime Calabi--Yau divisor,
 \item $V\not\sim -K_X$, \text{or}
 \item $c^2_1(X)\cdot c_2(X)\neq 0$
\end{enumerate}
holds, 
then the Cartier divisor $V-mK_X$ is ample for $m\gg 1$.
\end{thm}

As a corollary of Theorem \ref{thm.delta.4} and \cite{liu}*{Theorem 4.1}, we proved Campana--Peternell's conjecture in dimension 4 for most of the cases:

\begin{cor}\label{cor.cp4}
Let $X$ be a projective smooth fourfold such that $-K_X$ is strictly nef.
If $-K_X$ is not ample, then $c^2_1(X)\cdot c_2(X)=0$ 
and $-K_X\sim V$, where $V$ is a prime Calabi--Yau divisor. 
\end{cor}

These log version results show that the strictly nef divisor $L$ or 
$-K_X$ share a same property with ample divisors. Hence, if these strictly nef divisors are not ample,
then they provide a rather strong restriction on the position
and the shape of the pseudoeffective cone and the nef cone of $X$. 
We will see these restrictions in Subsections \ref{subsec3.2} and \ref{subsec4.2}. 

\subsection{Hirzebruch--Riemann--Roch formula}\label{subsec.hrr}
Let $X$ be a projective smooth fourfold such that $-K_X$ is nef. 
We can calculate Chern classes of $(\Omega^p_X)^{\otimes q}(-mK_X)$ for any positive integers $p,q, m$ by \cite{fulton}*{Examples 3.2.2 and 3.2.3} and the fact 
that 
\[
2\ch (E\wedge E)=(\ch (E))^2-r_2\cdot \ch (E),
\]
where $E$ is a vector bundle and $r_2$ acts on $H^{2k}(X, \mathbb R)$ by multiplying by $2^k$.
Throughout this paper, we mainly concern about the case $\chi(\mathcal O_X)=1$,
$c^4_1(X)=(-K_X)^4=0$ and $c^2_1(X)\cdot c_2(X)=0$. Under these assumptions,
the Hirzebruch--Riemann--Roch formula gives that
 \begin{equation}\label{eq.hrr1}
 \begin{aligned}
      \chi(X, \Omega^1_X(-mK_X))&=-\frac{6m+1}{12}c_1(X)\cdot c_3(X)-\frac{1}{6}c_4(X)+4\\
      \chi(X, \Omega^3_X(-mK_X))&=\frac{6m-1}{12}c_1(X)\cdot c_3(X)-\frac{1}{6}c_4(X)+4.
 \end{aligned}
 \end{equation}
We further assume that $c_1(X)\cdot c_3(X)=0$. Then, combining with this additional assumption, 
the Hirzebruch--Riemann--Roch formula gives that
 \begin{equation}\label{eq.hrr2}
\chi(X, \Omega^2_X(-mK_X))=\frac{2}{3}c_4(X)+6.
 \end{equation}
Finally, we assume that $\chi(\mathcal O_X)=1$ and $c^4_1(X)=c^2_1(X)\cdot c_2(X)=c_1(X)\cdot c_3(X)=c_4(X)=0$. Then,
the Hirzebruch--Riemann--Roch formula gives that
 \begin{equation}\label{eq.hrr3}
\chi(X, (\Omega^1_X)^{\otimes 2}(-mK_X))=c^2_2(X)+16.
 \end{equation}

\section{The canonical divisor is numerically trivial}\label{sec3}

In this section, we study Conjecture \ref{conj.amp} in dimension 3,
according to the singularities that $X$ contains.

\subsection{$X$ has non-canonical singularities} 
In this case, we allow $X$ to contain any kind of singularities worse than canonical singularities, and then show that the ampleness conjecture holds.
The proof is similar to Case 2 of \cite{liu-matsumura}*{Theorem 3.3}, so we sketch the proof here for the reader's convenience. 

\begin{thm}\label{thm.nc}
Let $X$ be a projective normal variety of dimension 3 such that $K_X$ is $\mathbb Q$-Cartier and numerically trivial. Assume that $X$ has non-canonical singularities. Then, any strictly nef $\mathbb Q$-Cartier divisors on $X$ are ample.
\end{thm}

\begin{proof}
Let $L$ be a  strictly nef $\mathbb Q$-Cartier divisor on $X$.
Let $f\colon Y\to X$ be a relative canonical model of $X$, where
$Y$ has canonical singularities and $K_Y$ is $f$-ample  
(see \cite{liu-matsumura}*{Subsection 2.3} for more details). 
Let $B$ be the $f$-exceptional effective $\mathbb Q$-divisor on $Y$ such that $K_Y+B=f^*K_X\equiv 0$. 
Since $X$ has non-canonical singularities, we have that $B\neq 0$.

By \cite{liu-matsumura}*{Lemma 2.4}, we obtain that
$K_Y+tf^*L$ is nef for $t\geq 6$.
If $K_Y+tf^*L$  is not big for any $t>6$, then 
\begin{equation}\label{eq.rs}
    K_Y^3=K_Y^2\cdot f^*L=K_Y\cdot (f^*L)^2= (f^*L)^3=0
\end{equation}
as in \cite{serrano}*{Lemma 1.3} or \cite{liu-matsumura}*{Theorem 3.3}. In particular, $B^3=-K_Y^3=0$, which implies that $\dim f(B)\geq 1$. 
Since $B$ is $f$-exceptional, we have that $\dim f(B)=1$.
Then as in  Case 2 of \cite{liu-matsumura}*{Theorem 3.3}, take a very ample divisor $H$ on $X$ such that $K_Y+f^*H$ is ample on $Y$.
Then, for a sufficiently large integer $k$,  
we see that $k(K_Y+f^*H)\cdot B$ is represented by a curve that is not contracted by $f$, so 
$k(K_Y+f^*H) \cdot B \cdot f^*L>0$ by the projection formula. 
Since $K_Y\cdot B \cdot f^*L=-K_Y^2\cdot f^*L=0$,
we obtain that 
$$
H \cdot f(B) \cdot L=f^*H \cdot B \cdot f^*L>0,
$$ 
which contradicts that $\dim f(B)=1$.

Therefore, $K_Y+tf^*L$  is big for $t>6$. 
It follows that $K_Y+B+tf^*L\equiv tf^*L$ is big on $Y$, and hence $L$ is big on $X$. Since $X$ is normal, the non-klt locus $\nklt (X)$ of $X$ is of dimension $\leq 1$. Hence, $L|_{\nklt (X)}$ is ample by the strict nefness of $L$. In particular, $mL|_{\nklt (X)}$ is generated by global sections for every $m\gg 1$. Then by the basepoint-free theorem 
(see \cite{fujino-foundations}*{Corollary 4.5.6}, or proved by the well-known X-method and the Nadel vanishing), we obtain that $L$ is ample.
\end{proof}

Immediately, we obtain the ampleness conjecture for threefolds with non-canonical singularities as a special case of Theorem \ref{thm.nc}.

\begin{cor}\label{cor.nc}
Let $X$ be a projective log canonical variety of dimension 3 such that $K_X$ is numerically trivial. Assume that $X$ has non-canonical singularities. Then, any strictly nef $\mathbb Q$-Cartier divisors on $X$ are ample.
\end{cor}

By the same spirit, we can also prove a very general result for the case where $-K_X$ is strictly nef in dimension 3.

\begin{thm}\label{thm.cp3}
 Let $X$ be a projective normal variety of dimension 3
such that $-K_X$ is $\mathbb Q$-Cartier and strictly nef. 
Assume that $X$ has non-canonical singularities. 
Then, $-K_X$ is ample.
\end{thm}

\begin{proof}
Put $L:=-K_X$.
Let $f\colon Y\to X$ be a relative canonical model of $X$, where
$Y$ has canonical singularities and $K_Y$ is $f$-ample. 
Let $B$ be the $f$-exceptional effective $\mathbb Q$-divisor on $Y$ such that $K_Y+B=f^*K_X=-f^*L$, 
where $B\neq 0$ since $X$ has non-canonical singularities. 
If $K_Y+tf^*L$  is nef but not big for any $t>6$, then as \eqref{eq.rs}, we obtain
\begin{equation}\label{eq.rs2}
    B^3=(-f^*L-K_Y)^3=0 \quad \text{and} \quad 
K_Y\cdot B \cdot f^*L=K_Y\cdot(-f^*L-K_Y)\cdot f^*L=0.
\end{equation}
The first equation in \eqref{eq.rs2} implies that $\dim f(B)=1$. As the proof of Theorem \ref{thm.nc}, for a very ample divisor $H$ on $X$ such that $K_Y+f^*H$ is ample,
we see that $(K_Y+f^*H) \cdot B \cdot f^*L>0$.
Then, the second equation in \eqref{eq.rs2} implies that
$H \cdot f(B) \cdot L=f^*H \cdot B \cdot f^*L>0$, which contradicts that $\dim f(B)=1$. 

Therefore, $K_Y+tf^*L$  is nef and big for $t>6$. 
It follows that $K_Y+B+tf^*L=(t-1)f^*L$ is big on $Y$, and hence $L$ is nef and big on $X$. Then as above, 
the basepoint-free theorem \cite{fujino-foundations}*{Corollary 4.5.6} and the strict nefness of $L$ imply that $L$ is ample.
\end{proof}

As a direct consequence, we prove Conjecture \ref{conj.cp} in dimension 3 completely. 

\begin{cor}\label{cor.cp3}
 Let $X$ be a projective normal variety of dimension 3
such that $-K_X$ is $\mathbb Q$-Cartier and strictly nef. 
Then, $-K_X$ is ample. In particular, Conjecture \ref{conj.cp} holds in dimension 3.
\end{cor}

\begin{proof}
By Theorem \ref{thm.cp3}, we can assume that $X$ has at worst canonical singularities. This case is proved by \cite{uehara}*{Theorem 3.8} (see \cite{loywz}*{Theorem D} for klt singularities). 
In another way, we have $\kappa(X, -K_X)\geq 0$ by \cite{lmptx}. Then, we can perform induction of the dimension onto the effective divisor $E\sim -mK_X$  for some $m>0$, exactly the same as what \cites{ccp, liu, liu-matsumura, serrano} and the references therein have done.
\end{proof}


These results support a philosophy that the worse singularities $X$ contains, the easier the log version of Serrano's conjecture becomes.

\subsection{$X$ has at worst canonical singularities}\label{subsec3.2}
In Conjecture \ref{conj.amp}, $X$ can be irregular, that is, $q(X):=h^1(X,\mathcal O_X)\neq 0$. 
In this case, after using \emph{Albanese mapping defined for varieties with rational singularities} instead, 
we can prove the ampleness conjecture exactly the same as \cite{ccp}*{Section 3}. 
This is done in an unpublished preprint by Chen Jiang and the author in 2020.
For more general results, see \cite{wz}*{Theorem 4.1, Corollary 4.2}.

\begin{thm}\label{thm.irregular}
Let $X$ be a projective canonical variety of dimension 3 with $K_X\equiv 0$ and $q(X)\geq 1$.
Then, any strictly nef $\mathbb Q$-Cartier divisors on $X$ are ample.
\end{thm}

\begin{proof}
Since $X$ is rational, by \cite{bs}*{Lemma 2.4.1},
there exists a commutative diagram as follows:
\[
\xymatrix{
\widetilde{X} \ar[d]_{\tau}\ar[r]^-\alpha &\alb(\widetilde{X})\\
X\ar[ur]_\beta&
}
\]
where $\tau$ is an arbitrary log resolution, $\alpha$ is the Albanese mapping and $\beta$ is a morphism called {\em{Albanese mapping defined for $X$}}.
 It is easy to see that 
 $\dim \alb(\widetilde{X})=q(\widetilde{X})=q(X)\geq 1$. 
Then, the proof of \cite{ccp}*{Theorem 3.1} works verbatim 
after replacing \cite{ccp}*{Proposition 1.6 (3)} by the abundance theorem for 
canonical Calabi--Yau threefolds and 
\cite{ccp}*{Theorem 1.5 (1)} by \cite{liu-matsumura}*{Theorem 3.1}.
\end{proof}


Hence, by Corollary \ref{cor.nc}, Theorem \ref{thm.irregular} and Lemma \ref{lem.Q}, we only need to consider the ampleness conjecture for canonical Calabi--Yau threefolds (see Definition \ref{defn.ccy}).
The following proposition is a generalization of \cite{liu-matsumura}*{Theorem 3.6} to $\mathbb Q$-factorial canonical Calabi--Yau threefolds. 

\begin{prop}\label{prop.sn}
Let $X$ be a $\mathbb Q$-factorial canonical Calabi--Yau threefold and $L$ be a strictly nef $\mathbb Q$-Cartier divisor on $X$.
Then, for any prime divisor $D$ on $X$, the $\mathbb Q$-Cartier divisor $D+tL$ is ample for $t\gg 1$. In particular, 
\begin{enumerate}
    \item if $\nu(X, L)=2$, then $L|_D$ is ample on any prime divisor $D$;
    \item if $\nu(X, L)=1$, then $L\cdot E_1\cdot E_2>0$ for any two effective divisors $E_1$ and $E_2$.
\end{enumerate}
\end{prop}

\begin{proof}
These conclusions follow from a small modification of \cite{liu-matsumura}*{Theorems 3.5 and 3.6} for  $\mathbb Q$-factorial canonical Calabi--Yau threefolds. We sketch the proof in the following. By considering the log canonical pair $(X, \delta D)$ for a small rational number $\delta>0$, 
we have that $D+tL$ is a strictly nef $\mathbb Q$-Cartier divisor on $X$ for $t\gg 1$ by \cite{liu-matsumura}*{Lemma 2.4}. 
Let $f\colon S\to D$ be the relative minimal model of the normalization of $D$. Then, there exists an effective $\mathbb Q$-divisor $B$ on $S$ such that $K_S+B=f^*D|_D$.
Since $f^*L|_D$ is almost strictly nef (see \cite{ccp} or \cite{liu-matsumura}*{Subsection 2.1} for its definition and properties),
$K_S+tf^*L|_D$ is big for $t\gg 1$ by \cite{liu-matsumura}*{Theorem 3.1}. It follows that $K_S+B+tf^*L|_D=f^*(D+tL)|_D$ is big, and hence 
$(D+tL)|_D$ is nef and big for $t\gg 1$.  Therefore, 
\[
(D+tL)^3=(D+tL)^2\cdot tL+(D+tL)^2\cdot D\geq (D+tL)^2\cdot D=(D|_D+tL|_D)^2>0,
\] that is,
$D+tL$ is strictly nef and big for $t\gg 1$.  By the basepoint-free theorem, 
$D+tL$ is ample for $t\gg 1$.  

Then, case (1) follows directly by the same proof of \cite{liu-matsumura}*{Theorem 3.6}.
For case (2), since $E_1+tL$ and $E_2+tL$ are ample for $t\gg 1$, $(E_1+tL) \cdot (E_2+tL)$ is represented by a movable $\mathbb Q$-effective curve on $X$. By the strict nefness of $L$, we obtain that 
\[
0<(E_1+tL) \cdot (E_2+tL) \cdot L=E_1\cdot E_2\cdot L+
t(E_1+E_2)\cdot L^2+t^2L^3.
\]
Since $\nu(L)=1$, we have that $L^3=L^2=0$. These imply that
$L\cdot E_1\cdot E_2>0$ for any two effective divisors $E_1$ and $E_2$.
\end{proof}

\begin{rem}\label{rem.factorial}
In Proposition \ref{prop.sn}, if we do not assume that $X$ is $\mathbb Q$-factorial, then $D$ is not necessarily $\mathbb Q$-Cartier, so it is meaningless to discuss the ampleness of $D+tL$. However, if we assume that $D$ (also $E_1$, $E_2$) is $\mathbb Q$-Cartier, then the proof of Proposition \ref{prop.sn} also works
without assuming that $X$ is $\mathbb Q$-factorial.
\end{rem}

If there indeed exists a strictly nef but not ample divisor on $X$, 
then Proposition \ref{prop.sn} provides 
 a rather strong restriction on the position
and the shape of the pseudoeffective cone and the nef cone in $H^2(X,
\mathbb R)$. For example, the corollary below shows that $c_2(X)$ is non-negative on the pseudoeffective cone of $X$ and
$c_2(X)$ is a limit of movable curves in
the sense of \cite{bdpp} by the duality of cones. 

\begin{cor}\label{cor.sn}
Let $X$ be a $\mathbb Q$-factorial terminal Calabi--Yau threefold.
If there exists a strictly nef but not ample $\mathbb Q$-Cartier divisor $L$ on $X$, then $c_2(X)\cdot D>0$ for any prime divisor $D$. 
\end{cor}

\begin{proof}
By \cites{ccp,serrano} (see also \eqref{eq.chi} and the proof of Lemma \ref{lem.Q}), we can assume that $L^3=c_2(X) \cdot L=0$.
If $c_2(X)=0$, then it is well-known that any nef divisor on $X$ is semiample (see \cite{oguiso-sakurai}*{Theorem 0.1 (IV)} and the references therein), contradicting that $L$ is strictly nef but not ample. 
Then, since $c_2(X)\neq 0$ is pseudoeffective by \cite{miyaoka}*{Theorem 6.6} and $D+tL$ is ample for $t\gg 1$ by Proposition \ref{prop.sn}, we obtain that
\[
c_2(X) \cdot D=c_2(X) \cdot (D+tL)>0
\]
for any prime divisor $D$ by Kleiman's ampleness criterion.
\end{proof}

\begin{rem}\label{rem.c2}
A \emph{$c_2$-contraction} is a nontrivial contraction $f\colon X\to Y$ satisfying $f^*H\cdot c_2(X)=0$ for an ample divisor $H$ on $Y$.
A fibration $\pi\colon X\to Y$ of type $\rm{\,I\,}_0$ or type $\rm{I\hspace{-.01em}I}_0$
in the sense of Oguiso \cite{oguiso}*{Main Theorem}
is a $c_2$-contraction.
Corollary \ref{cor.sn} implies that
if there exists a strictly nef but not ample divisor on a 
$\mathbb Q$-factorial terminal Calabi--Yau threefold $X$, then
there is no $c_2$-contraction on $X$, in particular, 
there is no fibration $\pi\colon X\to Y$ of type $\rm{\,I\,}_0$ or type $\rm{I\hspace{-.01em}I}_0$.
\end{rem}

In \cite{liu-svaldi}, we mainly focus on the case $\nu(X, L)=2$.
We can generalize \cite{liu-svaldi}*{Theorem 1.6} to the singular case, following \cite{lop}*{Proposition 8.13} 
(see also \cite{liu-svaldi}*{Lemma 4.2}):

\begin{thm}\label{thm.nu2.dim3}
Let $X$ be a $\mathbb Q$-factorial  canonical Calabi--Yau threefold and $L$ be a strictly nef $\mathbb Q$-Cartier divisor on $X$. 
If $\nu(X, L)=2$, then $H^p(X, \Omega^{[q]}_X(mL))=0$ for all $p, q\geq 0$ and $m\gg 1$, in particular, $\chi(\Omega^{[q]}_X(mL))=0$ for all $q$ and $m\gg 1$.
\end{thm}
\begin{proof}
Since $\nu(X, L)=2$, $L$ is not ample.
By Proposition \ref{prop.sn} (1) and \cite{liu-svaldi}*{Lemma 3.2}, we obtain that $H^p(X, \Omega^{[q]}_X(mL))=0$ for $p\geq 2$, $q\geq 0$ and $m\gg 1$.
Then, replacing \cite{lop}*{Theorem 8.12} by this result, the proof of 
 \cite{lop}*{Proposition 8.13} works verbatim.
\end{proof}

\begin{rem}
In the smooth case, \cite{lop}*{Equation (1)} predicted that any vanishing results in Theorem \ref{thm.nu2.dim3} can never happen; 
some evident is that,
$\chi(\Omega^{q}_X(mL))=\chi(\Omega^{[q]}_X(mL))=0$ will imply that $c_3(X)=0$,
as showed in \cites{lop, liu-svaldi}. However, in the singular case, the Hirzebruch--Riemann--Roch formula is not easy to calculate; in the canonical but not terminal case, we do not even know that $\hat c_2(X):=c_2(\Omega^{[1]}_X)$ is pseudoeffective or not.
\end{rem}

In the rest of this subsection, we investigate the case $\nu(X, L)=1$.
We show how absurd by the existence of a strictly nef divisor with numerical dimension 1.

\begin{prop}\label{prop.int}
Let $X$ be a $\mathbb Q$-factorial canonical Calabi--Yau threefold and $L$ be a strictly nef $\mathbb Q$-Cartier divisor on $X$. 
If $\nu(X, L)=1$,
then 
\begin{enumerate}
    \item $\dim D_1\cap D_2=1$ for any two different prime divisors $D_1$ and $D_2$, in particular, there are no disjoint prime divisors on $X$;
    \item $L\cdot \omega_D>0$ for any prime divisor $D$ on $X$, in particular, there is no surface with pseudoeffective anti-canonical divisor on $X$.
\end{enumerate}
\end{prop}

\begin{proof}
For case (1), if $\dim D_1\cap D_2<1$, then $L\cdot D_1\cdot D_2=0$, contradicting to Proposition \ref{prop.sn} $(2)$. 
For case (2),  we have that $\omega_D=(K_X+D)|_D=D|_D$ by adjunction. 
Therefore, $L\cdot \omega_D=L\cdot D^2>0$ 
by Proposition \ref{prop.sn} $(2)$ again.
\end{proof}

\begin{rem}\label{rem.hk}
Although we can not solve the ampleness conjecture completely for the case 
$\nu(X, L)=1$, Proposition \ref{prop.int} provides a practical criterion for most of the known canonical Calabi--Yau threefolds. For example,
there is no birational morphism on $X$ contracting a divisor to a point by Proposition \ref{prop.int} (2); 
there is also no birational morphism on $X$ contracting two divisors onto two curves disjoint with each other by 
Proposition \ref{prop.int} (1). If there are several birational contractions on $X$, their exceptional divisors must be highly involved.
\end{rem}

From different disciplines such as mirror symmetry and birational geometry, it is conjectured 
that, perhaps after some flopping, any Calabi--Yau threefold $X$ with $\rho(X)\geq 2$ should admit a fibered structure.
Some evident of this conjecture is that,  almost all the known Calabi--Yau threefolds with $\rho(X)\geq 2$ are fibered.
Assuming the existence of
a strictly nef divisor with numerical dimension 1, 
the above propositions and remarks show that the elliptic fibration (a fibration whose general fiber is of genus one) is the only possible fibered structure on a canonical Calabi--Yau threefold; 
moreover, the structure of this elliptic fibration is also very restrictive as follows. We will see similar restrictions in the next section.

\begin{thm}\label{thm.nu1.dim3}
Let $X$ be a $\mathbb Q$-factorial canonical Calabi--Yau threefold admitting a fibration $f\colon X\to S$.
Assume that there exists a strictly nef $\mathbb Q$-Cartier divisor $L$ on $X$ with $\nu(X, L)=1$.
Then 
\begin{enumerate}
    \item $f$ is an equiv-dimensional elliptic fibration;
    \item $S$ is a log terminal Fano surface with $\rho(S)=1$.
\end{enumerate}
\end{thm}
\begin{proof}
If $S$ is a curve, then any two general fibers of $f$ are disjoint with each other, contradicting to Proposition \ref{prop.int} (1).
Hence, $S$ is a surface and $f$ is an elliptic fibration.
By the canonical bundle formula, 
there exists an effective $\mathbb Q$-divisor $\Delta_S$ on $S$ such that $(S,\Delta_S)$ is log terminal and $K_S\sim_{\mathbb Q}-\Delta_S$;
in particular, $S$ is a log terminal surface, hence is $\mathbb Q$-factorial 
(see the statements right before \cite{oguiso}*{Lemma 3.4}).

If $f$ is not equiv-dimensional, then there exists a surface $D$ on $X$ mapping to a point $P$ of $S$. 
It follows that $D \cap f^*C=\emptyset$ for a curve $C$ on $S$ avoiding  $P$,  contradicting to Proposition \ref{prop.int} (1) again.
Hence, $f$ is equiv-dimensional.

Let $F$ be a general fiber of $f$ (an elliptic curve)
and $C$ be an effective curve on $S$.  By Proposition \ref{prop.sn} $(2)$, we obtain that 
\[
0<L\cdot (f^*C)^2=(C^2)L\cdot F.
\]
Since $L\cdot F>0$ by the strict nefness of $L$, we obtain that  $C^2>0$.
If $\Delta_S\neq 0$, then $K_S$ is not nef, so we can run the $K_S$-MMP.
Let $\pi\colon S\to T$ be an extremal contraction. If $\dim T=2$, then the self-intersection of the contracted curve is negative;
if $\dim T=1$, then the self-intersection of any fiber of $\pi$ is zero. Both cases contradict that $C^2>0$ for any curve $C$ on $S$.
Therefore, $T$ is a point, that is, $-K_S$ is ample and $\rho(S)=1$.

If $\Delta_S=0$, then $K_S\sim_{\mathbb Q} 0$. In this case,
we look deep into the canonical bundle formula given in \cites{ambro, fujino-mori, kawamata,oguiso} and the references therein. Note that 
$\Delta_S=B_S+M_S$, where $B_S$ is the effective discriminant part and $M_S$ is the nef moduli part. Since the general fiber is an elliptic curve, 
it is well-known that $M_S$ is semiample. Therefore, $\Delta_S=0$ means that $B_S=0$ and $M_S\sim_{\mathbb Q}0$. 
Let $\pi\colon S'\to S$ be the index one cover induced by $M_S$. Actually now $\pi$ is an \'etale double cover 
from a canonical K3 surface $S'$ onto a canonical Enriques surface $S$,
as showed in the proof of  \cite{oguiso}*{Lemma 3.4}.
Let $f'\colon X'\to S'$ be the base change of $f$ induced by $\pi$.
Then $X'=S'\times F$, as the discriminant part $B_{S'}$ and the moduli part $M_{S'}$ are trivial,
and $\rho\colon X' \to X$ is an \'etale double cover.
 However, $\rho^*L$ is a strictly nef divisor on $X'$,
 hence is ample on $X'=S'\times F$ obviously. This contradicts that $\nu(X, L)=1$. 
\end{proof}

 \begin{rem}
If we further assume that $X$ is simply connected and terminal in Theorem \ref{thm.nu1.dim3}, then by  \cite{oguiso}*{Lemma 3.4}, 
there is no need to discuss the case $\Delta_S=0$ in above proof.
\end{rem}

\section{The anti-canonical divisor is strictly nef}\label{sec4}

In this section, we improve the results on Campana--Peternell's conjecture in dimension 4. 
By \cite{loy}*{Theorem 1.2}, we can assume that the smooth fourfold $X$ in Conjecture \ref{conj.cp} is rationally connected throughout this section.
By Corollary \ref{cor.cp4}, the remaining case is that
$X$ is a projective rationally connected smooth fourfold such that $c_1^2(X)\cdot c_2(X)=0$ and 
$-K_X\sim V$ is strictly nef, where $V$ is a prime Calabi--Yau divisor. 

\begin{prop}\label{prop.kd0}
 Let $X$ be a projective smooth fourfold such that 
$-K_X\sim V$ is strictly nef, where $V$ is a prime Calabi--Yau divisor. If  $\kappa(X, -K_X)\geq 1$, then $-K_X$ is ample.
\end{prop}
    
\begin{proof}
By \cite{loy}*{Theorem 1.2}, we can assume that $X$ is rationally connected.
Let $-mK_X=\sum r_iE_i$ be the irreducible decomposition for some $m\geq 1$. If there exists some $E_i$, say $E_1$, such that $E_1\not\sim -K_X$, then $E_1-tK_X$ is ample by Theorem \ref{thm.delta.4} and $\sum_{i\geq 2} r_iE_i-tK_X$ is strictly nef by \cite{liu}*{Lemma 2.5}  for $t\gg 1$. 
It follows that 
\[
-mK_X-(r_1+1)tK_X=r_1(E_1-tK_X)+(\sum_{i\geq 2} r_iE_i-tK_X)
\]
is ample for $t\gg 1$. That is, $-K_X$ is ample. So we assume that every integral component of $-mK_X$ for every $m\geq 1$ is linearly equivalent to $-K_X$.
Since $\kappa(X, -K_X)\geq 1$, there must be some $m\geq 1$ and some
 integral component $E (\sim -K_X)$ of $-mK_X$ such that $E\neq V$.
In particular, $h^0(X, \mathcal O_X(-K_X))\geq 2$.
 Consider the exact sequence
 \[
 0\to \mathcal O_X\to  \mathcal O_X(-K_X)\to \mathcal O_V(-K_X)\to 0.
 \]
Since $h^1(X, \mathcal O_X)=0$, we obtain that $h^0(V, \mathcal O_V(-K_X))=h^0(X, \mathcal O_X(-K_X))-h^0(X, \mathcal O_X)\geq 1$.
That is, $-K_X|_V$ is a strictly nef and effective divisor on $V$, 
hence is ample by the abundance theorem for canonical Calabi--Yau threefolds. In particular, $(-K_X)^4=(-K_X|_V)^3>0$. 
Then, by the basepoint-free theorem, we obtain that $-K_X$ is ample.
\end{proof}

Since $\kappa(X, -K_X)\geq 0$ by \cite{liu}*{Theorem 1.4}, we immediately obtain the following implication of Proposition \ref{prop.kd0}. 

\begin{cor}\label{cor.kd0}
Let $X$ be a projective smooth fourfold such that $-K_X$ is strictly nef. If $-K_X$ is not ample, then 
$\kappa(X, -K_X)=0$ and $\nu(X, -K_X)=2$ or $3$.
\end{cor}

\begin{proof}
We only need to exclude the case $\nu(X, -K_X)=1$. 
Note that $-K_X\sim V$, where $V$ is a prime Calabi--Yau divisor, as stated at the beginning of this section. 
If $\nu(X, -K_X)=1$, then $\nu(V, -K_X|_V)=0$, contradicting that 
$-K_X|_V$ is strictly nef on $V$.
\end{proof}

\subsection{The case $\nu(X, -K_X)=3$}
We show that Campana--Peternell's conjecture holds in this case. More precisely, we show the following:

\begin{thm}\label{thm.nm3}
Let $X$ be a projective smooth fourfold such that $-K_X$ is strictly nef. 
If $-K_X$ is not ample, then $\nu(X, -K_X)=2$.
\end{thm}

\begin{proof}
As stated at the beginning of this section, 
we assume that $X$ is rationally connected, 
$c_1^2(X)\cdot c_2(X)=0$ and 
$-K_X\sim V$, where $V$ is a prime Calabi--Yau divisor. 
In particular, $\chi(\mathcal O_X)=1$.
For a contradiction, we assume that $\nu(X, -K_X)=3$ 
by Corollary \ref{cor.kd0}. It follows that $\nu(V, -K_X|_V)=2$ by definition.

Let $H$ be a very ample and general hypersurface on $X$.
Then $H\cap V=H|_V$ is an irreducible surface by Bertini's theorem,
which is Cartier on $V$ by viewing as a divisor.
Let $S$ be an irreducible surface on $H$. 
If $S=H\cap V$, then $-K_X|_S$ is ample on $S$ by Proposition \ref{prop.sn} (1) and Remark \ref{rem.factorial}.
If $S\neq H\cap V$, then $\dim S\cap V\leq 1$. Assume that
$\dim S\cap V=0$. Then $S\cdot V\cdot H=0$. 
However, replacing $H$ by a multiple if necessary, we can see that
$S\cdot H$ is represented by a movable $\mathbb Q$-effective curve on $S$. Hence, $V\cdot (S\cdot H)>0$ by the strict nefness of $V$, 
which is a contradiction. 
Therefore, we obtain that $\dim S\cap V=1$, and hence
$(-K_X|_S)^2=V^2\cdot S=V\cdot (S\cdot V)>0$. 
Note that $(-K_X|_H)^3=(-K_X)^3\cdot H>0$ 
(see \cite{liu-svaldi}*{Lemma 2.1} for example).
By the Nakai--Moishezon criterion,  $-K_X|_H$ is ample.

Then by \cite{liu-svaldi}*{Lemma 3.2}, we obtain that 
\begin{equation}\label{eq.vanishing}
H^i(X, \Omega^q_X(-mK_X))=0 \quad \text{and} \quad
H^i(X, (\Omega^1_X)^{\otimes 2}(-mK_X))=0  
\end{equation}
for $1\leq q\leq 3$, $i\geq 2$ and $m\gg 1$.
If $c_1(X)\cdot c_3(X)\neq 0$, then combining with the Hirzebruch--Riemann--Roch formula \eqref{eq.hrr1} and \eqref{eq.vanishing}, 
we  obtain
\begin{equation}
    h^0(X, \Omega^q_X(-mK_X))\geq \chi(X,  \Omega^q_X(-mK_X))\geq 5
\end{equation}
for $q=1$ or $3$, and $m\gg 1$. So as Step 2 of \cite{lp}*{Theorem 5.1}, there exists a positive integer $r$ and a Cartier divisor $N$ such that 
$h^0(X, N)\geq 2$ and $\mathcal O_X(N+rmK_X)$ is a subsheaf saturated in $\bigwedge^r \Omega^q_X$ for $q=1$ or $3$. 
By the duality of \cite{ou}*{Theorem 1.4} (see also \cite{liu}*{Theorem 2.2}),
there exists a pseudoeffective divisor $F$ such that $-rmK_X=N+F$.
Since $\kappa(N)\geq 1$, $N\not\sim -K_X$ by Corollary \ref{cor.kd0}.
Hence, $N-tK_X$ is ample for $t\gg 1$ by Theorem \ref{thm.delta.4}.
It follows that 
\[
(rm+t)(-K_X)^4=(-K_X)^3\cdot (-rmK_X-tK_X)=(-K_X)^3\cdot (N-tK_X)+(-K_X)^3\cdot F>0,
\]contradicting that $\nu(X, -K_X)=3$.

Therefore, we assume that $c_1(X)\cdot c_3(X)= 0$. In this case, the Hirzebruch--Riemann--Roch formulas 
\eqref{eq.hrr1} and \eqref{eq.hrr2} provide that
\begin{equation}
    \begin{aligned}
    h^0(X, \Omega^1_X(-mK_X))\geq& \chi(X,  \Omega^1_X(-mK_X))=
    -\frac{1}{6}c_4(X)+4 \quad \text{and}\\
   h^0(X, \Omega^2_X(-mK_X))\geq& \chi(X,  \Omega^2_X(-mK_X))=\frac{2}{3}c_4(X)+6.
    \end{aligned}
\end{equation}
If $c_4(X)\neq 0$, then either $h^0(X, \Omega^1_X(-mK_X))\geq 5$ or $h^0(X, \Omega^2_X(-mK_X))\geq 7$. The same as above, these
contradict to Corollary \ref{cor.kd0}. 
Therefore, we further assume that $c_4(X)=0$.
Then, the analytic Euler characteristic 
\[
1=\chi(\mathcal O_X)=-\frac{1}{720}(c^4_1(X)-4c_1^2(X)\cdot c_2(X)
-c_1(X)\cdot c_3(X)-3c^2_2(X)+c_4(X))
\]
gives that $c^2_2(X)=240$. By the Hirzebruch--Riemann--Roch formula 
\eqref{eq.hrr3}, we obtain
\begin{equation}
    h^0(X, (\Omega^1_X)^{\otimes 2}(-mK_X))\geq \chi(X, (\Omega^1_X)^{\otimes 2}(-mK_X))=c^2_2(X)+16=256.
\end{equation}
Again, the same as above, this
contradicts to Corollary \ref{cor.kd0}.
\end{proof}

\begin{cor}\label{cor.nm3}
Let $X$ be a projective smooth fourfold such that $-K_X$ is strictly nef. Assume that $X$ admits a fibration 
$f\colon X\to C$ onto a curve $C$. Then, $-K_X$ is ample.
\end{cor}

\begin{proof}
Let $F$ be a general fiber of $f$. Then, $-K_F=-K_X|_F$ is strictly nef,
which implies that $-K_F$ is ample
by \cite{serrano}*{Theorem 3.9}. In particular,
$(-K_X)^3\cdot F=(-K_X|_F)^3=(-K_F)^3>0$, and hence
$\nu(-K_X)\geq 3$. Then, our conclusion follows from Theorem \ref{thm.nm3}.
\end{proof}

\subsection{The case $\nu(X, -K_X)=2$}\label{subsec4.2}
In this case, we present some results similar to 
Proposition \ref{prop.sn} (2) and Theorem \ref{thm.nu1.dim3}.

\begin{prop}\label{prop.sn4}
Let $X$ be a projective smooth fourfold such that
$-K_X\sim V$ is strictly nef, where $V$ is a prime Calabi--Yau divisor. 
If $\nu(X, -K_X)=2$, then $(-K_X)^2\cdot E_1\cdot E_2>0$ for any two prime divisors $E_1\neq V$ and $E_2\neq V$.
\end{prop}

\begin{proof}
Since $E_1-tK_X$ and $E_2-tK_X$ are ample for $t\gg 1$ by 
Theorem \ref{thm.delta.4},
$(E_1-tK_X)|_V \cdot (E_2-tK_X)|_V$ is represented by a movable 
$\mathbb Q$-effective curve on $V$. 
It follows that 
\[
0<(-K_X)|_V\cdot (E_1-tK_X)|_V \cdot (E_2-tK_X)|_V = (-K_X)^2\cdot (E_1-tK_X) \cdot (E_2-tK_X).
\]
Since $\nu(X, -K_X)=2$, we obtain that 
\[
(-K_X)^2\cdot (E_1-tK_X)\cdot (E_2-tK_X)= (-K_X)^2\cdot E_1\cdot E_2.
\]
Hence, our conclusion follows.
\end{proof}

Consider the $K_X$-MMP. Let $g\colon X\to Y$ 
be an extremal contraction induced by some extremal ray. 
Then, $g$ is one of the following three types:
\begin{enumerate}
    \item $g\colon X\to Y$ is a Fano contraction;
    \item $g\colon X\to Y$ is a divisorial contraction;
    \item $g\colon X\to Y$ is a small contraction.
\end{enumerate}
In the rest of this subsection, we mainly focus on  
the Fano contraction $g \colon X\to Y$ induced by some extremal ray. 
If $\dim Y=0$, then $-K_X$ is ample directly by the cone theorem.
If $\dim Y=1$, then $-K_X$ is ample by Corollary \ref{cor.nm3}.
So we assume that $\dim Y\geq 2$. 

\begin{thm}\label{thm.fc}
Let $X$ be a projective smooth fourfold such that $-K_X$ is strictly nef.
Assume that $X$ admits a Fano contraction $g \colon X\to Y$ 
onto a surface $Y$ induced by some extremal ray. 
If $\nu(X, -K_X)=2$, then $\rho(X)=2$ and $Y$ is a log terminal 
Fano surface with $\rho(Y)=1$.
\end{thm}

\begin{proof}
As before, we 
assume that $X$ is rationally connected, 
$c_1^2(X)\cdot c_2(X)=0$ and 
$-K_X\sim V$, where $V$ is a prime Calabi--Yau divisor. 
Note that $V$ is $g$-horizontal by the strict nefness of $V$.
By the minimal model program, $Y$ is a $\mathbb Q$-factorial rationally connected log terminal surface.
Let $F$ be a general fiber of $g$ 
and $C$ be an effective curve on $Y$.  
By Proposition \ref{prop.sn4}, we obtain that 
\[
0<(-K_X)^2\cdot (f^*C)^2=(C^2)(-K_X)^2\cdot F=
(C^2)(-K_X)\cdot V\cdot F.
\]
Since $V$ is $g$-horizontal, we have that $V\cap F$ is a curve, 
and hence $(-K_X)\cdot V\cdot F>0$ by the strict nefness of $-K_X$.
It follows that  $C^2>0$.
Then, we run the $K_Y$-MMP.
Let $\pi\colon Y\to T$ be an extremal contraction. If $\dim T=2$, then the self-intersection of the contracted curve is negative;
if $\dim T=1$, then the self-intersection of any fiber of $\pi$ is zero. Both cases contradict that $C^2>0$ for any curve $C$ on $Y$.
Therefore, $T$ is a point, that is, $-K_Y$ is ample and $\rho(Y)=1$.
Since $g$ is an extremal contraction induced by some extremal ray, $\rho(X)=\rho(Y)+1=2$.
\end{proof}

Note that in all extremal contractions with $\dim Y\geq 2$,  
we can consider the exact sequence
\[
0\to \mathcal O_X(K_X) \simeq \mathcal O_X(-V) \to  \mathcal O_X  \to \mathcal O_V \to 0
\]
and its pushforward by $g$:
\begin{equation}\label{eq.long.exact}
  \begin{aligned}
0&\to g_* \mathcal O_X(K_X) \to  g_*\mathcal O_X\simeq \mathcal O_Y  
\to g_*\mathcal O_V \to R^1 g_* \mathcal O_X(K_X) \to R^1 g_*\mathcal O_X\to R^1 g_*\mathcal O_V \\
&\to R^2 g_* \mathcal O_X(K_X) \to R^2 g_*\mathcal O_X\to R^2 g_*\mathcal O_V \to 0.
\end{aligned}  
\end{equation}
Here $R^i g_*\mathcal O_X(K_X)$ is torsion free for any $i\geq 0$ by \cite{kollar}*{Theorem 2.1}, but $R^i g_*\mathcal O_X$ is not torsion free in general. 
Let $H$ be a very ample divisor on $Y$. Then $g^*H-tK_X$ is ample for $t\gg 1$ by Theorem \ref{thm.delta.4}. Since $g^*H$ is semiample,
$t(g^*H-K_X)=(t-1)g^*H+(g^*H-tK_X)$ is ample for $t\gg 1$.
In particular, $g^*H-K_X$ is ample.
It follows that 
\[
H^i(X, \mathcal O_X(g^*H))=H^i(X, \mathcal O_X(K_X+g^*H-K_X))=0
\]
for $i\geq 1$ by the Kodaira vanishing. 
Then, by \cite{kollar-mori}*{Proposition 2.69}
and Grothendieck's duality, we obtain that
\begin{equation}\label{eq.le.van}
R^i g_*\mathcal O_X=0  \text{ for } i\geq 1 \quad \text{and} \quad R^i g_*\mathcal O_X(K_X)=0  \text{ for } i<\dim F.
\end{equation}

 \begin{rem}
In the case $\dim Y=2$, 
the equations \eqref{eq.long.exact} and \eqref{eq.le.van} give that $\mathcal O_Y \simeq g_*\mathcal O_V$.
That is, $g|_V$ is connected by the Stein factorization and 
$g|_V\colon V\to Y$ is an elliptic fibration. 
The assumption $\nu(X, -K_X)=2$ implies that $\nu(V, -K_X|_V)=1$. 
In this case, a similar result has been showed in Theorem \ref{thm.nu1.dim3}.
However, the prime Calabi--Yau divisor $V$ in the proof of
Theorem \ref{thm.fc}
is not necessarily $\mathbb Q$-factorial. 
Therefore, we use the same technique to prove the similar result again,
instead of using Theorem \ref{thm.nu1.dim3}  directly.
 \end{rem}
 
 \begin{rem}
In the case $\dim Y=3$,
 we have a contraction
 $g\colon X\to Y$ onto a rationally connected threefold $Y$
such that $V$ is $g$-horizontal by the strict nefness of $V$.
Then, \eqref{eq.long.exact},
\eqref{eq.le.van} and Grothendieck's duality give that 
\begin{equation}\label{eq.d3}
    0\to \mathcal O_Y \to g_*\mathcal O_V \to \mathcal O_Y(K_Y)\to 0
\quad \text{and} \quad R^ig_*\mathcal O_V=0 \text{ for } i\geq 1.
\end{equation}
It follows that $g_*\mathcal O_V=\mathcal O_Y\oplus  \mathcal O_Y(K_Y)$
and $g|_V\colon V\to Y$ is a double covering. These 
conditions  provide strong restrictions on $V$ and $Y$;
moreover,  after pulling back onto $V$, we can see that $\dim D_1\cap D_2=1$ for any two $\mathbb Q$-Cartier divisors $D_1$ and $D_2$ on $Y$ by
Proposition \ref{prop.int} (1) and Remark \ref{rem.hk}.
\end{rem}

\end{document}